\documentclass[12pt,reqno]{amsart}
\usepackage[final]{graphicx}
\usepackage{amsfonts}
\usepackage{pdfsync}
\usepackage{amsmath}
\usepackage{amssymb}
\usepackage{amsthm}

\usepackage{amsbsy}
\usepackage{enumerate}
\usepackage[T1]{fontenc}
\usepackage[utf8]{inputenc}
\usepackage{lmodern}
\usepackage{color}
\usepackage{comment}

\newtheorem{theorem}{Theorem}[section]
\newtheorem{lemma}[theorem]{Lemma}

\newtheorem{remark}[theorem]{Remark}

\newtheorem*{remarks}{Remarks}

\DeclareMathOperator{\arctanh}{arctanh}
\newcommand*\diff{\mathop{}\!\mathrm{d}}

\def\ii{\'\i}

\specialcomment{notes}{\begingroup\sffamily\color{blue}}{\endgroup}

\includecomment{comments}

\begin{document}

\title[Non-existence result for a radial Br\'ezis-Nirenberg problem]{\textbf{A non-existence result for a generalized radial Br\'ezis-Nirenberg problem}}

\author[Benguria]{Rafael D. Benguria$^1$}\thanks{The work of R.B.  has been supported by Fondecyt (Chile) Project \# 116--0586, and
 the Millenium Science Initiative of the Ministry of Economy, Development, and Tourism (Chile), grant {\it Nucleus for Cardiovascular Magnetic Resonance}.}
\author[Benguria]{Soledad Benguria$^2$}

\address{$^1$ Rafael~D.~Benguria, Instituto de F\ii sica, Pontificia Universidad Cat\'olica de Chile, Avda. Vicu\~na Mackenna 4860, Santiago, Chile}
\email{rbenguri@fis.puc.cl}
\address{$^2$ Soledad Benguria, Mathematics Department, University of Wisconsin - Madison, 480 Lincoln Dr., Madison, WI 53703, USA.}
\email{benguria@math.wisc.edu}


\smallskip

\begin{abstract}
We develop a new method for estimating the region of the spectral parameter of a generalized Br\'ezis--Nirenberg problem
for which there are no, non trivial, smooth solutions. 
This new method combines the standard Rellich--Pohozaev argument with a Hardy type inequality for bounded domains. 
The estimates we get are better than the usual estimates for low dimensions. 
\end{abstract}



\maketitle

\maketitle

\section{Introduction.}

As pointed by us in \cite{BeBe17}, virial theorems have played, for a long time, a key role in the localization of linear and nonlinear eigenvalues. 
In the spectral theory of Schr\"odinger Operators, the virial theorem has been widely used to prove  the absence of positive eigenvalues for various 
multiparticle quantum systems (see, e.g., \cite{We67,Si74,Ba75}). 
In 1983,  Br\'ezis and Nirenberg  \cite{BrNi83} considered the existence and nonexistence of solutions  of the nonlinear equation 
$$
-\Delta u = \lambda u + |u|^{q-1} u, 
$$
defined on a bounded, smooth domain of $\mathbb{R}^n$, $n >2$, with Dirichlet boundary conditions, where $q=(n+2)/(n-2)$ is the critical Sobolev exponent. 
In particular, they used a virial theorem, namely the Pohozaev identity \cite{Po65}, to prove the nonexistence of regular  solutions when the domain is star--shaped, 
for any $\lambda \le 0$, in any $n >2$. After the classical paper \cite{BrNi83} of Br\'ezis and Nirenberg, many people have considered extensions of this problem in different settings. In 
particular, the Br\'ezis--Nirenberg (BN) problem has been studied on bounded, smooth, domains of the hyperbolic space $\mathbb{H}^n$ (see, e.g., \cite{St02,BaKa12,GaSa14,Be16}), 
where one replaces the Laplacian by the Laplace--Beltrami operator in $\mathbb{H}^n$. Stapelkamp \cite{St02} proved the analog of the above mentioned nonexistence result of 
Br\'ezis--Nirenberg in $\mathbb{H}^n$. Namely she proved that there are no regular solutions of the BN problem for bounded, smooth,  star--shaped domains in $\mathbb{H}^n$ ($n>2$), 
if $\lambda \le n(n-2)/4$. 

The purpose of this manuscript is to develop a new method for estimating the region of the spectral parameter of a generalized Br\'ezis--Nirenberg problem, 
for which there are no, non trivial, smooth solutions. 
This new method combines the standard Rellich--Pohozaev argument with a Hardy type inequality for bounded domains. 
The estimates we get are better than the usual estimates for low dimensions. 

\bigskip

Here we consider a generalized radial Br\'ezis-Nirenberg problem, which is given through the following 
boundary value problem. Given $R>0$, we are interested in estimating the region of the spectral parameter 
parameter $\lambda$ for which there are no non trivial smooth (more precisely $u \in C^2[0,R]$) solutions of
\begin{equation}\label{eq:BNG}\begin{cases}
-u''(x) - (n-1)\dfrac{a'(x)}{a(x)}u'(x) = \lambda u(x)+ |u(x)|^{q-1}u, \\
u'(0) = u(R) = 0, \\
\end{cases}\end{equation}
where $q= (n+2)/(n-2)$ is the critical Sobolev constant, $n>2,$ and $a \in C^3[0,R]$ satisfies

\bigskip

\begin{enumerate}[i)]\label{def:a}
\item{$a(0) = 0;$}
\item{$a'(x)> 0$ for all $x\in (0,R)$; and}
\item{there exists $\omega \ge 0$ such that $a''(x) \ge \omega a(x)$ for all $x \in (0,R).$ }
\end{enumerate}

\bigskip
\bigskip
\noindent
The standard procedure to determine the region of the parameter $\lambda$ for which there are no solutions to 
(\ref{eq:BNG}) is to write the equation in terms of a {\it conformal} second order operator and then use the 
Rellich--Pohozaev technique \cite{Re40,Po65}. See, e.g., \cite{St02,Be16} for the determination of the 
range of the parameter $\lambda$ for which there are no solutions of the equivalent of (\ref{eq:BNG}) 
in the case of $\mathbb{H}^n$, i.e., when $a(x)=\sinh(x)$. Applying this procedure we can prove the following. 

\bigskip
\bigskip

\begin{theorem}\label{thm1}
Problem (\ref{eq:BNG}) has no non--trivial solution if 
\begin{equation}
\lambda \le \mu^*(n,R) \equiv  \frac{n-2}{4} \, \inf_{0<x<R}\left\{ (n-1)\frac{a''}{a} + \frac{a'''}{a'}\right\}.
\label{eq:stap}
\end{equation}
\end{theorem}

\bigskip

\begin{remarks}
i) In the case of geodesic balls in $\mathbb{H}^n$, $n \ge 2$, which corresponds to having $a(x)=\sinh(x)$ in (\ref{eq:BNG}) 
we recover the result of Theorem 1, ii) of \cite{St02}. In that case there are no solutions of (\ref{eq:BNG}) if 
$$
\lambda \le \mu^*(n,R) = \frac{n(n-2)}{4}.
$$
ii) If one restricts to the case of positive solutions, one can prove non existence when $\lambda \le 1+ \pi^2/(16 \arctanh^2 R)$ for
$n=3$ (see, \cite{St02} Theorem 2). See also Theorem 1.1 in \cite{Be16} for the analogous result in the hyperbolic case for $2 \le n \le 4$.
iii) Also, if we restrict to the case of positive solutions, by a standard argument one does not have solutions if $\lambda \ge \lambda_1(n,R)$, 
where $\lambda_1(n,R)$ is the first eigenvalue of the operator, 
$$
-\frac{d^2}{dx^2} - (n-1)\dfrac{a'(x)}{a(x)} \, \frac{d}{dx}, 
$$
on $[0,R]$, with boundary conditions  $u'(0) = u(R) = 0$. 
\end{remarks}

\bigskip
\bigskip
As pointed above, the purpose of this manuscript is to develop a different approach to prove non existence of solutions. In fact we develop an scheme that 
combines the use of the Rellich--Pohozaev technique together with the use of Hardy type inequalities on bounded domains. Using this new scheme we 
can prove the following.

\bigskip
\bigskip

\begin{theorem}\label{thm2}
Problem (\ref{eq:BNG}) has no non--trivial solution if 
\begin{equation}
\lambda \le \lambda^*(n,R) \equiv \frac{n(n-1)}{4} \, C,
\label{eq:new}
\end{equation}
where 
\begin{equation}\label{eq:C}
C = \min\left\{ \frac{D+2\omega}{2(n+2)}, \frac{D}{4}\right\},
\end{equation}
and 
\begin{equation}\label{eq:T}
D = \inf_{0<x<R}\left\{ (2n-3)\frac{a''}{a} + \frac{a'''}{a'}\right\}.
\end{equation}

\end{theorem}

The rest of the manuscript is organized as follows. In section 2 we express the problem (\ref{eq:BNG}) 
in terms of an appropriate operator reminiscent of the conformal laplacian. Once we have that, we apply 
the standard Rellich--Pohozaev technique \cite{Re40, Po65} to prove Theorem \ref{thm1}. In Section 3 we set
forward our new approach. We start with the problem (\ref{eq:BNG}) stated as such. This,  in the $\mathbb{H}^n$ case,  
would correspond to studying radial solutions of the  BN problem for geodesic balls, in terms of geodesic coordinates
centered at the center of the balls. In that setting we use a combination of the Rellich--Pohozaev technique \cite{Re40, Po65}
together with a Hardy type inequality to prove Theorem \ref{thm2}. Finally, in Section 4 we illustrate and compare the bounds obtained 
in \ref{thm1} and \ref{thm2} through a particular example. The bound obtained through our new technique is better than the standard 
bound for a whole region of the parameter space $(n,R)$, in particular for all $R$ when $n \le 4$.

\section{Non-existence of solutions, via the Pohozaev virial identity, using the conformal laplacian.}

A key step towards proving the lower bound $\lambda > n(n-2)/4$ for the existence of solutions to the Br\'ezis--Nirenberg 
problem on a bounded domain in the Hyperbolic space $\mathbb{H}^n$ is to use the so called ``stereographic'' projection. 
By using the stereographic projection one writes the Laplace--Beltrami $\Delta_{\mathbb{H}^n}$ in terms of the conformal laplacian
in Euclidean coordinates. In fact, one can write, 
$$
\Delta_{\mathbb{H}^n} u = p^{-n} {\rm div} \left(p^{n-2} \, \nabla u \right), 
$$
where $p=2/(1+|x|^2)$. 
Here we  mimic what is done in $\mathbb{H}^n$ for our generalized operator
\begin{equation}
\frac{d^2}{d\theta^2} + (n-1) \frac{a'(\theta)}{a(\theta)} \frac{d}{d\theta} 
\label{eq:SA1}
\end{equation}
and write this as the radial part of a conformal laplacian 
\begin{equation}
p^{-n} {\rm div} \left(p^{n-2} \, \nabla ( \cdot ) \right)
\label{eq:SA2}
\end{equation}
written in euclidean coordinates, for some appropriate function $p$. Given the function $a(\theta)$ that defines our generalized operator (\ref{eq:SA1}), 
our first goal is to determine the conformal factor $p(r)$ in our case.
In order to do that we need to make the change of variables 
$\theta \to r(\theta)$, with $r= \vert x \vert$ in such a way that 
\begin{equation}
H(u) \equiv u'' + (n-1) \frac{a'(\theta)}{a(\theta)} u' = p^{-n} {\rm div} \left(p^{n-2} \, \nabla u \right) \equiv L(u). 
\label{eq:SA3}
\end{equation}
Since here we have radial symmetry, 
\begin{equation}
L(u) = p^{-n} r^{1-n} \frac{d}{dr} \left(r^{n-1} p^{n-2} \frac{du}{dr} \right).
\label{eq:SA4}
\end{equation}
Recall that in radial  coordinates, 
$$
{\rm div} \vec F = \nabla \cdot \vec F = \frac{1}{r^{n-1}} \frac{d}{dr}\left(r^{n-1} F_r \right).
$$
If we denote by $\dot u= du/dr$, we can write $L(u)$ as
\begin{equation}
L(u) = \frac{1}{p^2} \left[ \ddot u + \left(\frac{(n-1}{r} + (n-2) \frac{\dot p}{p} \right) \dot u \right]. 
\label{eq:SA5}
\end{equation}
On the other hand, If we denote by $u'=du/d\theta$, and use the chain rule we have
$$
u' = \dot u \frac{dr}{d\theta},
$$
and
$$
u'' = \ddot u \left(\frac{dr}{d\theta} \right)^2+  \dot u \,  \frac{d^2 r}{d \theta^2}.
$$
Hence, 
\begin{equation}
H(u) = \ddot u  \left(\frac{dr}{d\theta}\right)^2 + \left(\frac{d^2 r}{d \theta^2}+ (n-1) \frac{a'(\theta)}{a(\theta)}  \frac{dr}{d\theta}     \right) \dot u.
\label{eq:SA6}
\end{equation}
Replacing $H$, given by (\ref{eq:SA6}), and $L$,  given by (\ref{eq:SA5}) in (\ref{eq:SA3}), and comparing coefficients, we conclude at once that
\begin{equation}
\left( \dfrac{d r}{d \theta}\right) = \dfrac{1}{p},
\label{eq:SA7}
\end{equation}
and also that
\begin{equation}
\frac{d^2 r}{d \theta^2}+ (n-1) \frac{a'(\theta)}{a(\theta)}  \frac{dr}{d\theta}  = \frac{1}{p^2} \left(\frac{n-1}{r}+ (n-2) \frac{\dot p}{p} \right).
\label{eq:SA8}
\end{equation}
Differentiating (\ref{eq:SA7}) with respect to $\theta$ and again using the chain rule we have 
\begin{equation}
\frac{d^2 r}{d \theta^2}= - \frac{\dot p}{p^2} \, \frac{dr}{d\theta}= - \frac{\dot p}{p^3}.
\label{eq:SA9}
\end{equation}
Using  (\ref{eq:SA7}) and (\ref{eq:SA9}) in (\ref{eq:SA8}) we conclude, 
\begin{equation}
\frac{a'}{a} = \frac{1}{p \, r} + \frac{\dot p}{p^2} \equiv B(r). 
\label{eq:SA10}
\end{equation}
For the change of variables $\theta \to r(\theta)$  to be well defined we need $p>0$. From (\ref{eq:SA10}) and (\ref{eq:SA7}) is simple to obtain 
$r(\theta)$ and $p(\theta)$. In fact, $\dot p = p' \, d \theta/dr= p' \, p$. Since we also have $(1/p) = r'$, from (\ref{eq:SA10}) we get, 
\begin{equation}
\frac{a'}{a} = \frac{r'}{r} + \frac{p'}{p},
\label{eq:SA11}
\end{equation}
which can be immediately integrated to yield, 
\begin{equation}
a=p \, r, 
\label{eq:SA12}
\end{equation} 
(here we have chosen an irrelevant  integrating constant to be $1$). However, $p=1/r'$. So, from (\ref{eq:SA12}) we have 
$$
a(\theta) = \frac{r}{r'},
$$
which can be integrated to yield, 
$$
r(\theta) = \exp \int_{\theta_0}^{\theta} (1/a(s)) \, ds.
$$

\bigskip
\bigskip

\begin{proof} [Proof of Theorem \ref{thm1}]
\noindent
Once we have done the transformation from our original variable $\theta$ to $r$, via the relation with the conformal laplacian, we are ready to apply the standard 
Rellich--Pohozaev \cite{Re40,Po65} argument. This is done with the purpose of determining the region of $\lambda$ for which there are no non trivial solutions of the BN problem. 
For $n \ge 3$ the equation we consider is 
\begin{equation}
-L(u) = -  p^{-n} r^{1-n} \frac{d}{dr} \left(r^{n-1} p^{n-2} \frac{du}{dr} \right) = \lambda u + |u|^{4/(n-2)}\, u. 
\label{eq:SA13}
\end{equation}
We now introduce the change of the dependent variable $u \to v$ given by
\begin{equation}
u = p^{1-(n/2)} \, v.
\label{eq:SA14}
\end{equation}
After some long and  straightforward computations one can rewrite (\ref{eq:SA13}) as
\begin{equation}
- \ddot v - \frac{n-1}{r} \, \dot v + V \, v = \lambda p^2 \, v + |v|^{4/(n-2)} \, v, 
\label{eq:SA15}
\end{equation}
where the {\it potential} $V$ is given by
\begin{equation}
V= (n-2) \, \left[\frac{1}{2} \frac{\ddot p}{p} + \frac{1}{4} (n-4) \left(\frac{\dot p}{p} \right)^2 + \frac{1}{2} (n-1)  \frac{\dot p}{r \, p} \right].
\label{eq:SA16}
\end{equation}
Using (\ref{eq:SA10}) the potential $V$ can be expressed in terms of $B$ as, 
\begin{equation}
V= (n-2) \, \left[\frac{n}{4} p^2 B^2 + \frac{1}{2} \, p \dot B - \frac{(n-2)}{4 r^2}\right].
\label{eq:SA17}
\end{equation}
If we multiply (\ref{eq:SA15}) by $v$ we obtain, 
\begin{equation}
-\frac{1}{r^{n-1}}\frac{d}{dr} \left(r^{n-1} v  \, \dot v \right) + {\dot v}^2  + V\, v^2= \lambda \, v^2 + |v|^{4/(n-2)} \, v^2.
\label{eq:SA17.1}
\end{equation}
Now we proceed with the Rellich--Pohozaev argument applied to the equation (\ref{eq:SA15}). Multiply (\ref{eq:SA15})
by $r \, \dot v$ (here we are using that $r \, d/dr$ is the generator of radial dilations).
Among other quantities we need the identity, 
\begin{equation}
r \, \dot v \, V \, v = \frac{1}{r^{n-1}} \frac{d}{dr} \left(r^n \frac{v^2}{2} \, V\right) - \frac{n}{2} v^2 V - r  \frac{v^2}{2} \dot{V}.
\label{eq:SA18}
\end{equation}
Using (\ref{eq:SA17}) and the fact that
\begin{equation}
\frac{\dot p}{p}+\frac{1}{r} = p \, B, 
\label{eq:SA18.1}
\end{equation}
after some simple computations we conclude that 
\begin{equation}
\frac{1}{2} r \dot V + V = \frac{(n-2)}{4} r \, B \, p^3  \, T,
\label{eq:SA19}
\end{equation}
where $T$ is given, in terms of the function $B$, by
\begin{equation}
 T \equiv n \, B^2 + (n+1) \frac{\dot B}{p} + \frac{\dot B}{B} \frac{1}{p^2 \, r} + \frac{\ddot B}{B \, p^2}.
\label{eq:SA20}
\end{equation}
This particular expression involving the potential $V$ as well as the definition of $T$ will prove to be useful later.

\bigskip
\noindent
Moreover we  need the identity, 
\begin{equation} 
r \, \dot v \, \ddot v + (n-1) {\dot v}^2 = \frac{(n-2)}{2}  {\dot v}^2 + \frac{1}{r^{n-1}} \frac{d}{dr} \left( \frac{1}{2} r^n {\dot v}^2 \right).
\label{eq:SA21}
\end{equation}
And the last identity we need to have is the following, 
\begin{eqnarray} 
r \, \dot v \, \left(\lambda \, v \, p^2 + |v|^{4/(n-2)} \, v\right)  &=
 \frac{1}{r^{n-1}} \frac{d}{dr} \left( r^n \left(\frac{\lambda}{2} p^2 \, {v}^2  + \frac{(n-2)}{2n} |v|^{4/(n-2)} \, v^2\right)\right) \nonumber \\
 & - \lambda \frac{v^2}{2} n \, p^2 - \lambda \, r \, v^2 p \, \dot p - \frac{(n-2)}{2} |v|^{4/(n-2)} \, v^2.
 \label{eq:SA22}
\end{eqnarray}
With all these identities at hand we can recapitulate. Multiplying (\ref{eq:SA15}) by $r \dot v$ and using (\ref{eq:SA18}), (\ref{eq:SA21}), and (\ref{eq:SA22})
we get, 
\begin{eqnarray}
\frac{1}{r^{n-1}} \frac{d}{dr} \left( r^n \left(\frac{1}{2} v^2 V- \frac{1}{2}{\dot v}^2 - \frac{\lambda}{2} p^2 \, {v}^2  - \frac{(n-2)}{2n} |v|^{4/(n-2)} \, v^2 \right)\right)\nonumber \\
=  \frac{(n-2)}{2}  {\dot v}^2  + \frac{n}{2} v^2 V + r  \frac{v^2}{2} \dot{V}  - \lambda \frac{v^2}{2} n \, p^2 - \lambda \, r \, v^2 p \, \dot p - \frac{(n-2)}{2} |v|^{4/(n-2)} \, v^2.
 \label{eq:SA23}
\end{eqnarray}
Multiplying (\ref{eq:SA17.1}) by $(n-2)/2$ and adding (\ref{eq:SA23}) to the result we obtain, 
\begin{eqnarray}
\frac{1}{r^{n-1}} \frac{d}{dr} \left( r^n \left(\frac{1}{2} v^2 V- \frac{1}{2}{\dot v}^2 - \frac{(n-2)}{2} \frac{v \, \dot v}{r} - \frac{\lambda}{2} p^2 \, {v}^2  - \frac{(n-2)}{2n} |v|^{4/(n-2)} \, v^2 \right)\right)\nonumber \\
=   v^2 \, \left(V + \frac{1}{2}  r \dot V \right)   - \lambda \, v^2 \, p^2 \, r \left(\frac{\dot p}{p} + \frac{1}{r} \right).
 \label{eq:SA24}
\end{eqnarray}
Finally, multiplying (\ref{eq:SA24}) by $r^{n-1}$ and integrating the result from $0$ to $R$, we obtain the following {\it virial} identity, 
\begin{equation}
\frac{1}{2} R^n \, {\dot v (R)}^2 = \int_0^R  v^2 \, p^3 \, B \left[\lambda - \frac{(n-2)}{4} \, T \right] r^n \, dr.
\label{eq:SA25}
\end{equation}
To obtain  (\ref{eq:SA25}) we used the boundary conditions on $v$, in particular that $v(R)=0$. We also used (\ref{eq:SA18.1}) and (\ref{eq:SA19}) 
in order to express the right side of (\ref{eq:SA25}) in terms of $B$ and $T$.

\bigskip
\noindent
Using the chain rule and the definition of $p$ we have that
$$
\dot B = B' \, p,
$$
and, 
$$
\ddot B = B'' \, p^2 + B' \, \dot p. 
$$
Replacing these two expressions in the definition  (\ref{eq:SA20}) of $T$ we get,
\begin{equation}
T = n \, B^2 + (n+1) \, B' + \frac{B''}{B} + \frac{B'}{B} \left(\frac{\dot p}{p^2} + \frac{1}{r \, p} \right) = n \, B^2 + (n+1) \, B' + \frac{B''}{B} + B',
\label{eq:SA26}
\end{equation}
where we used (\ref{eq:SA18.1}) to obtain the second equality. 
Using the fact that $B=a'/a$, after some algebra starting from (\ref{eq:SA26}), we finally conclude that
\begin{equation}
T = (n-1) \frac{a''}{a} + \frac{a'''}{a'}. 
\label{eq:SA27}
\end{equation}
Inserting this expression for $T$ back in the {\it virial} identity (\ref{eq:SA25})
we have
\begin{equation}
\frac{1}{2} R^n \, {\dot v (R)}^2 = \int_0^R  v^2 \, p^3 \, B \left[\lambda - \frac{(n-2)}{4} \, \left((n-1) \frac{a''}{a} + \frac{a'''}{a'} \right)  \right] r^n \, dr.
\label{eq:SA28}
\end{equation}
Since the right side of (\ref{eq:SA28}) is positive, we conclude that if 
\begin{equation}
\lambda  \le  \mu^*(n,R) \equiv \frac{(n-2)}{4} \inf_{r \in [0,R]}  \left((n-1) \frac{a''}{a} + \frac{a'''}{a'} \right), 
\label{eq:SA29}
\end{equation}
then there is no solution of (\ref{eq:BNG}). This concludes the proof of our Theorem \ref{thm1}.
\end{proof}

\bigskip

\section{Non-existence of solutions using the Rellich--Pohozaev argument and a Hardy type inequality}

In this section we use a combination of the Rellich--Pohozaev technique together with a Hardy type inequality for bounded intervals in order to prove
Theorem \ref{thm2}. Hardy type inequalities, among other things, play an important role in the analysis of Partial Differential Equations
(see., e.g., \cite{BrMa97, Da99,OpKu90}, and references therein).   We begin with the following Lemma.

\bigskip
\begin{lemma}\label{lem:boundlambda}

Let $u \in C^2[0,R]$ be a solution of (\ref{eq:BNG}), and let $G(x) = \displaystyle\int_0^x a^{n-1}(s)\, ds$, and 
$S(x) = \dfrac{G(x)a'(x)}{a(x)} - \dfrac{a(x)^{n-1}}{n}.$ Then, 

\begin{equation}\label{eq:L2}
\lambda \ge \frac{n(n-1)}{4} \left(\frac{\displaystyle \int_0^R {u'}^2(x) S(x)\, \diff x}{\displaystyle \int_0^R \frac{G(x)^2 u'(x)^2}{G'(x)}\, \diff x} \right). 
\end{equation}

\end{lemma}

\begin{proof}

\bigskip
\noindent
Suppose $u \in C^2[0,R]$ is a solution of (\ref{eq:BNG}). Multiplying (\ref{eq:BNG}) by $u \, a^{n-1}$ and integrating we obtain, after integrating the first term by parts, 

\begin{equation}\label{eq:P1}
\int_0^R {u'}^2 a^{n-1}\, \diff x = \lambda \int_0^R u^2 a^{n-1}\, \diff x + \int_0^R |u|^{q+1} a^{n-1}\, \diff x. 
\end{equation}
On the other hand, multiplying equation (\ref{eq:BNG}) by $u'G,$ where $G(x) = \displaystyle\int_0^x a^{n-1}(s)\, ds$ we obtain 
\begin{equation}\label{eq:P2}
-\int_0^R \left( \frac{{u'}^2}{2}\right)'G\, \diff x - (n-1) \int_0^R {u'}^2 \frac{Ga'}{a}\, \diff x = \lambda \int_0^R \left( \frac{u^2}{2}\right)' G\, \diff x + \int_0^R \left(\frac{|u|^{q+1}}{q+1}\right)' G\, \diff x.
\end{equation}
After integrating by parts, and since $G'(x) = a^{n-1}(x) $ and $G(0) = 0,$ equation (\ref{eq:P2}) reads
\begin{equation}\label{eq:P22}
-\frac{u'(R)^2G(R)}{2}+\int_0^R {u'}^2 \frac{a^{n-1}}{2}\, \diff x - (n-1) \int_0^R {u'}^2 \frac{Ga'}{a}\, \diff x = -\lambda \int_0^R u^2 \frac{a^{n-1}}{2}\, \diff x - \int_0^R \frac{|u|^{q+1}}{q+1}a^{n-1}\, \diff x. 
\end{equation}
Now, solving for the term in $|u|^{q+1}$ in equation (\ref{eq:P1}) and substituting this term in equation (\ref{eq:P22}) we obtain
\begin{equation}\label{eq:P}
\int_0^R {u'}^2 \left( \frac{a^{n-1}}{2} + \frac{a^{n-1}}{p+1}- (n-1) \frac{Ga'}{a}\right)\, \diff x + \int_0^R u^2 a^{n-1} \left( \frac{\lambda}{2} - \frac{\lambda}{q+1}\right) \diff x = \frac{u'(R)^2G(R)}{2}.
\end{equation} 
However, since $1/2 + 1/(q+1) = (n-1)/n$ and $1/2 - 1/(q+1) = 1/n$, we can write
\begin{equation}\label{eq:P3}
\frac{\lambda}{n} \int_0^R u^2 a^{n-1}\, \diff x = \frac{u'(R)^2G(R)}{2} +(n-1) \int_0^R {u'}^2 \left( \frac{Ga'}{a} - \frac{a^{n-1}}{n}\right)\, \diff x.
\end{equation}
By hypothesis, we have that if $x>0$ then $a(x)>0.$ Therefore, 
\begin{equation}\label{eq:La}
\lambda \ge \frac{ n(n-1)\displaystyle\int_0^R {u'}^2\left( \frac{Ga'}{a} - \frac{a^{n-1}}{n}\right)\, \diff x}{\displaystyle\int_0^R u^2 a^{n-1}\, \diff x}.
\end{equation} 
Now let $S(x) = Ga'/a - a^{n-1}/n,$ the coefficient of ${u'}^2$ of the integral in the numerator. Then $L\ge 0$ if $x>0.$ In fact, let $m(x) = Ga' - a^n/n$. Since $G(0) = 0,$ one has that $m(0) = 0.$ Also,  since $G'(x) = a^{n-1}(x),$ we have that $m'(x) = Ga''$. In particular, since by hypothesis $a'' \ge \omega a \ge 0$, we have that $m'(x) \ge 0.$ It follows that $m\ge 0$ for all $x\in(0,R),$ and since $a$ is positive in this range, $L \ge 0.$ 

\bigskip 


\bigskip
\noindent
We will now use a Hardy type inequality to rewrite the  integral in the denominator in terms of ${u'}^2.$ Integrating by parts, we can write
\begin{equation}\label{eq:H1}
\int_0^R u^2 G'\, \diff x =-2\int_0^R u \, u' G\, \diff x = -2 \int_0^R \left( u a^\frac{n-1}{2}\right)\left( Gu' a^\frac{1-n}{2}\right)\, \diff x.
\end{equation} 
By Cauchy--Schwarz, it follows that 
\begin{equation}\label{eq:CS1}
\left( \int_0^R u^2 a^{n-1}\, \diff x\right)^2 < 4 \int_0^R u^2 a^{n-1}\, \diff x \int_0^R \frac{G^2{u'}^2}{a^{n-1}}\, \diff x.
\end{equation}
\noindent 
where the conditions on $u$ require the above inequality to be strict. Thus, 
 \begin{equation}\label{eq:CS2}
\int_0^R u^2 a^{n-1}\, \diff x < 4 \int_0^R \frac{G^2{u'}^2}{a^{n-1}}\, \diff x. 
\end{equation}
Hence, using that $a^{n-1}(x) = G'(x)$, it follows from equations (\ref{eq:La}) and (\ref{eq:CS2}) that 
\begin{equation}
\lambda > \frac{n(n-1)}{4} \left(\frac{\displaystyle \int_0^R u'(x)^2 S(x)\, \diff x}{\displaystyle \int_0^R \frac{G(x)^2 u'(x)^2}{G'(x)}\, \diff x} \right),
\end{equation} 
which proves the lemma.
\end{proof}

\begin{lemma}\label{lem:bounds}
$S(x) \ge C\dfrac{G^2(x)}{G'(x)},$ where $ C$ is given by equation  (\ref{eq:C}). 
\end{lemma}

\begin{proof}
Let $f(x) = S(x) G'(x) - CG(x)^2.$ We need to show that $f \ge 0.$ As before, we write $S(x) = m(x)/a(x)$, with $m(x) = G(x)a'(x)-a(x)^n/n \ge 0$. Then $f(x) = a^{n-2}(x)m(x) - CG(x)^2.$ Since $a(0) = G(0) = 0,$ it follows that $f(0) = 0.$ Thus, it suffices to show that $f' \ge 0.$ 

\bigskip 

\noindent
We have that $f'(x) = a^{n-3}(x) g(x)$, where 
\begin{equation}
g(x) \equiv  (n-2) a'(x)m(x)+G(x)a(x) (a''(x)-2Ca(x)).
\label{eq:defg}
\end{equation}
Notice that $g(0) = 0$ so, in order to prove that $g(x) \ge 0$,  it suffices to show that $g'(x)\ge 0$. 
Differentiating (\ref{eq:defg}), we can write
\begin{equation}\label{eq:g1}
g'(x) = (2n-3) \,  G \, a' \, a''  + \left( \frac{2}{n}\right)a^n \, a'' - 2\,C\, a^{n+1} - 4 \, C \, G \, a \, a'+G\, a\,a'''. 
\end{equation}
Since by hypothesis $a \ge 0$ and $a' \ge 0,$ it follows from equation (\ref{eq:T}) that 
$$
D \, a \, a' \le (2n-3)\, a' \, a''+a \,a''',
$$
so we can write  
$$ g'(x)  \ge G \, a \, a'  \, (D -4C) + \frac{2a^na''}{n}-2C a^{n+1}.$$
Furthermore, since by hypothesis $a''\ge \omega a,$ it follows that 
$$g'(x) \ge Gaa' (D -4C)+ \frac{a^{n+1}}{n}(2\omega -2Cn).$$ 
However, since $m\ge 0,$ we have that $Ga' \ge a^n/n.$ In particular, if $C \le D/4,$ we have that 
$$g'(x) \ge \frac{a^{n+1}}{n} ( D -4C  +2\omega  -2Cn).$$ 
It follows that $g'\ge 0$ provided that $C \le \dfrac{D+2\omega}{2(n+2)}.$

\bigskip 

\noindent
Now, choosing 
\begin{equation*}
C = \min\left\{ \frac{D+2\omega}{2(n+2)}, \frac{D}{4}\right\},
\end{equation*}
it follows that $S(x) \ge C\dfrac{G^2(x)}{G'(x)}$, which proves the lemma.
 \end{proof}
 

\bigskip
\bigskip
It follows from Lemmas \ref{lem:boundlambda} and \ref{lem:bounds} that if $u$ is a solution of (\ref{eq:BNG}) then 
$$
 \lambda > \frac{n(n-1)}{4} \left(\frac{\displaystyle \int_0^R  S(x) u'(x)^2\, \diff x}{\displaystyle \int_0^R \frac{G(x)^2 u'(x)^2}{G'(x)}\, \diff x} \right) \ge \frac{Cn(n-1)}{4}.
 $$
Hence, we conclude that if 
\begin{equation}\label{eq:L3}
\lambda \le \lambda^*(n,R) \equiv \frac{Cn(n-1)}{4}, 
\end{equation}
then problem (\ref{eq:BNG}) has no  non trivial solution. This concludes the proof of Theorem \ref{thm2}.

\section{An illustrative example}

To compare the bounds $\mu^*(n,R)$ and $\lambda^*(n,R)$ embodied in theorems \ref{thm1} and \ref{thm2} above, it is instructive to work a 
specific example as an application. Consider 
\begin{equation}
a(x) = x e^x.
\label{eq:3.1}
\end{equation}
Then, $a'(x)=e^x (1+x)$, $a''(x)=e^x (2+x)$, and $a'''(x)=e^x (3+x)$. 
Define,
$$
f(x)\equiv (n-1) \frac{a''}{a} +  \frac{a'''}{a'} = \left(1+\frac{2}{x}\right) (n-1) + \frac{3+x}{1+x}.
$$ 
The function $f(x)$ is strictly decreasing, therefore
\begin{equation}
\inf_{0<x<R} f(x) =  f(R) = n + 2 \frac{n(1+R)-1}{R(1+R)}. 
\label{eq:3.2}
\end{equation}
Then, according to (\ref{eq:SA29}), the problem (\ref{eq:BNG}) with $a(x) = x e^x$ does not have a solution if 
\begin{equation}
\lambda \le \mu^*(n,R) \equiv \frac{1}{4} (n-2) \left( n + 2 \frac{n(1+R)-1}{R(1+R)} \right).
\label{eq:3.3}
\end{equation}
On the other hand, define
$$
g(x)\equiv (2n-3) \frac{a''}{a} +  \frac{a'''}{a'} = \left(1+\frac{2}{x}\right) (2n-3) + \frac{3+x}{1+x}.
$$ 
Again, the function $g(x)$ is strictly decreasing, therefore
\begin{equation}
D\equiv \inf_{0<x<R} g(x) =  g(R) =   \left(1+\frac{2}{R}\right) (2n-3) + \frac{3+R}{1+R}.
\label{eq:3.4}
\end{equation}
Also, we have, 
\begin{equation}
\omega  \equiv  \inf_{0<x<R} \frac{a''}{a} = \inf_{0<x<R} \left(1 + \frac{2}{x} \right) = 1 + \frac{2}{R}.
\label{eq:3.5}
\end{equation}
Thus, we have on the one hand, 
$$
D+2 \omega = \left(1+\frac{2}{R}\right) (2n-1) + \frac{3+R}{1+R}.
$$
and, on the other hand, 
$$
\frac{D}{4} =\frac{1}{4}  \left\{\left(1+\frac{2}{R}\right) (2n-3) + \frac{3+R}{1+R}\right\}.
$$
If we denote by $s=(1+R)(2+R)$, after some simple calculations we conclude that, $(D+2\omega)/(2n+4) \le D/4$ provided, 
\begin{equation}
n \ge {\hat n}(s) =\frac{1}{2} \left[\left(1+\frac{1}{s} \right) + \sqrt{\left(1+\frac{1}{s} \right)^2+8}\right]
\label{eq:3.6}
\end{equation}
Notice that ${\hat n}(s)$ is a decreasing function of $s$, such that $\hat n(2) = (3+\sqrt{41})/4 =2.35078\dots$ 
and $\lim_{s \to \infty} {\hat n}(s) = 2$. 

\bigskip
\noindent
Hence, if $n \ge {\hat n}((1+R)(2+R))$, 
\begin{equation}
\lambda^*(n,R)= \frac{n(n-1)}{8(n+2)} \, \left\{ \left(1+\frac{2}{R}\right) (2n-1) + \frac{3+R}{1+R} \right\}.
\label{eq:3.7}
\end{equation}

\bigskip
\noindent
Let us compare our new type of bound, $\lambda^*(n,R)$ with the more standard one $\mu^*(n,R)$ for this example. In order to do that, 
we need to determine
for what values of $n$ and $R$, 
\begin{equation}
\lambda^*(n,R) \ge \mu^*(n,R).
\label{eq:3.8}
\end{equation}
Using (\ref{eq:3.3}) and (\ref{eq:3.7}) one can check that (\ref{eq:3.8})
holds, if and only if,
\begin{equation}
F(n,R) \equiv 2 \,(R^2+3R)(4-n)n + 2 (8-n)(n-1)= 2 \left\{(s-2)(4-n)n +  (8-n)(n-1) \right\} \ge 0.
 \label{eq:3.9}
 \end{equation}
In general this holds for all $n \ge 4$, independently of $R$. As a function of $s=(R+1)(R+2)$, (\ref{eq:3.9}) holds 
provided 
\begin{equation}
n \ge {\tilde n}(s) = \frac{1}{2(s-1)} \left[(4s+1) + \sqrt{16 \, s^2 -24 \, s +33} \,\right]
\label{eq:3.1}
\end{equation} 
It is not hard to show that ${\tilde n}(s)$ is a decreasing function of $s$ for all $s>1$. In fact ${\tilde n}$ 
decreases from $8$, for $R=0$ (i.e., $s=2$) to $4$ when $R\to \infty$ (i.e., $s\to \infty$). 
Moreover, it is a simple excercise, which we leave to the reader, to prove that $D/4 \ge \mu^*(n,R)$, so 
in the case $2 \le n < \hat n(s)$, $\lambda^*(n,R) >  \mu^*(n,R)$, for all $R$. 
Hence, in the $(n,R)$ parameter space,  our bound 
is better provided $2 \le n < {\tilde n} (s)$, where $s=(R+1)(R+2)$.

\bigskip

\begin{remark}
For the radial hyperbolic case, i.e., for problem (\ref{eq:BNG}) with $a(x) = \sinh(x)$, it follows from theorems \ref{thm1} and 
\ref{thm2} that $\mu^*(n,R)= n(n-2)/4$, whereas $\lambda^*(n,R) = n^2(n-1)/4(n+2)$, independent of $R$. In that case 
$\lambda^*(n,R)$ is better than $\mu^*(n,R)$ for all $2 \le n<4$, and all $R>0$. This case was reported by us in \cite{BeBe17}.
\end{remark}

\bigskip

\section{Acknowledgements}

One of us (RB) would like to thank the organizers of the semester program {\it Spectral Methods in Mathematical Physics} for  their hospitality
at the Mittag--Leffler Institute while this manuscript was being completed.


\begin{thebibliography}{7}

\bibitem{Ba75}
E.~Balslev, Absence of positive eigenvalues of Schr\"odinger Operators, 
{\it Archive~Rational~Mechanics and Applications}, {\textbf 59} (1975), 343--357. 


\bibitem{BaKa12}
C.~Bandle and Y.~Kabeya, On the positive, radial solutions of a semilinear elliptic equation in $\mathbb{H}^N.$ {\it Adv. Nonlinear Anal.,} {\textbf 1} (2012), 1--25.

\bibitem{Be16}
S. Benguria, The solution gap of the Br\'ezis-Nirenberg problem on the hyperbolic space. {\it Monatsh. Math.}, {\textbf 181} (2016), 537--559.

\bibitem{BeBe17}
R.~D.~Benguria and S.~Benguria, An improved bound for the non-existence of radial solutions of the Br\'ezis-Nirenberg problem in $\mathbb{H}^n$. 
 {\it Functional Analysis and Operator Theory for Quantum Physics, A Festschrift in Honor of Pavel Exner}, J. Dittrich, H. Kovarik, A. Laptev (Eds.), 
 Europ. Math. Soc. Publ. House, 2017, pp. 153--160.  

\bibitem{BrMa97}
H.~Br\'ezis and M.~Marcus,  Hardy's Inequalities Revisited. {\it Ann. Scuola Norma. Sup. Pisa Cl. Sci.} (4),  {\textbf XXV } (1997), 217--237.


\bibitem{BrNi83}
H.~Br\'ezis and L.~Nirenberg, Positive solutions of nonlinear elliptic equations involving critical Sobolev exponents. {\it Comm. Pure Appl. Math.}, {\textbf 36} (1983), 437--477.

\bibitem{Da99}
E.~B.~Davies, A review on Hardy inequalities.
{\it Operator Theory Advances and Applications}, {\textbf 110}, pp.~55--67, Birkh\"auser Verlag, Basel, 1999.

\bibitem{GaSa14}
D.~Ganguly and K.~Sandeep, Sign changing solutions of the Br\'ezis-Nirenberg problem in
the hyperbolic space. {\it Calc. Var. Partial Differential Equations}, {\textbf 50} (1-2) (2014), 69--91.

\bibitem{OpKu90}
B.~Opic and A.~Kufner, Hardy--type inequalities. 
Pitman Research Notes in Math.,
{\textbf 219}, Longman, 1990.


\bibitem{Po65}
S.~I.~Pohozaev, On the eigenfunctions of the equation $\Delta u+\lambda f(u)=0. $ {\it Dokl. Akad. Nauk.}, {\textbf  165} (1965), 36--39. 

\bibitem{Re40}
F.~Rellich, Darstellung der Eigenwerte von $\Delta u + \lambda  u = 0$ durch ein Randintegral, {\it Math. Z.}, {\textbf 46} (1940), 635--636.

\bibitem{Si74}
B.~Simon, Absence of positive eigenvalues in a class of multiparticle quantum systems, 
{\it Math.~Ann.}, {\textbf 207} (1974), 133--138.


\bibitem{St02}
S.~Stapelkamp, The Br\'ezis-Nirenberg problem on $\mathbb{H}^n.$ Existence and uniqueness of solutions.  {\it Elliptic
and parabolic problems} (Rolduc/Gaeta, 2001), World Sci. Publ., River Edge, NJ 2002, 283--290.

\bibitem{We67}
J.~Weidmann, The virial theorem and its application to the spectral theory of Schr\"odinger operators, 
{\it Bull.~Amer.~Math.~Soc.} {\textbf 77} (1967), 452--456.


\end{thebibliography}
\end{document}